\documentclass[12pt,reqno]{amsart}
 \usepackage[usenames,dvipsnames]{color}
\usepackage{graphicx}
\usepackage{fullpage}
\usepackage{amsthm}
\usepackage{shuffle}
\usepackage{float}
\usepackage[mathscr]{euscript}
\usepackage[colorlinks=true,
   citecolor=cyan,urlcolor=blue]{hyperref}

\usepackage[french]{todonotes}

\usepackage{wrapfig}
\usepackage{amssymb}
\usepackage{amsmath}

\def\rouge{\textcolor{red}}
\def\bleu{\textcolor{blue}}

\makeatletter
\newtheorem*{rep@theorem}{\rep@title}
\newcommand{\newreptheorem}[2]{
\newenvironment{rep#1}[1]{
\def\rep@title{#2 \ref{##1}}
\begin{rep@theorem}}
{\end{rep@theorem}}}
\makeatother
\def\auteur#1{{\sc #1}}
\def\titreref#1{{\em #1}}
\def\vol#1{{\bf #1}}

\newtheorem{lemma}{\bleu{Lemma}}
\newtheorem{conj}{\bleu{Conjecture}}

\newtheorem{prop}{\bleu{Proposition}}

\newreptheorem{thm}{Theorem}

\numberwithin{equation}{section}
\numberwithin{thm}{section}
\numberwithin{lemma}{section}
\numberwithin{rmk}{section}
\numberwithin{prop}{section}
\numberwithin{cor}{section}
\numberwithin{defn}{section}
\numberwithin{alg}{section}

\def\petit#1{\hbox{\scriptsize$#1$}}
\def\mbf#1{{\mathbf #1}}

\def\sign{\,\ \mathrm{sign}}

 \def\Q#1{(h_2-e_2)^{#1}}

\renewcommand{\S}{\mathbb{S}}

\newcommand{\fix}{\operatorname{fix}}

\def\qfoulkes#1#2#3#4{\langle [#1,#2]\!:\![#3,#4]\rangle_q}
\def\foulkes#1#2#3#4{\langle [#1,#2]\!:\![#3,#4]\rangle}
\def\define#1{\bleu{\bf{#1}}}

\def\GL{\mathrm{GL}}
\def\N{\mathbb{N}}
\def\C{\mathbb{C}}

\def\pref#1{{\rm (\ref{#1})}}

\def\Rat{\mathbb{Q}}
\newcommand{\scalar}[2]{{\langle#1,#2 \rangle}}
\def\charac{\raise 2pt\hbox{\large$\chi$}}

\newdimen\carrelength

\def\bleu{\textcolor{blue}}
\def\rouge{\textcolor{red}}

\title[$q$-Foulkes]{A $q$-Analog of Foulkes Conjecture}
\author{F.~Bergeron}
\address{\href{http://bergeron.math.uqam.ca}{D\'epartement de Math\'ematiques, Lacim, UQAM.}}
  \email{\href{mailto:bergeron.francois@uqam.ca}{bergeron.francois@uqam.ca}}
  \date{\rouge{\bf March, 2016}. This work was supported by NSERC}

\begin{document}


\begin{abstract} We propose a $q$-analog of classical plethystic  conjectures due to Foulkes. In our conjectures, a divided difference of plethysms of Hall-Littlewood polynomials $H_n(\mbf{x};q)$ replaces the analogous difference of plethysms of complete homogeneous symmetric functions $h_n(\mbf{x})$ in Foulkes conjecture. At $q=0$, we get back the original statement of Foulkes, and we show that our version holds at $q=1$.  We discuss further supporting evidence, as well as various generalizations.
\end{abstract}

\maketitle
 \parskip=0pt
{ \setcounter{tocdepth}{1}\parskip=0pt\footnotesize \tableofcontents}
\parskip=8pt  
\parindent=20pt


\section{Introduction} 
The aim of this text is to present and discuss a new $q$-analog of a conjecture due to Foulkes in his paper of 1949  (see~\cite{foulkes}). Recall that this classical conjecture states that the difference of plethysms\footnote{This operation is due to Littlewood~\cite{littlewood}, and its definition is recalled in an upcoming section.} 
    $$f_{a,b}:=h_b[h_a]- h_a[h_b],$$
 of homogeneous symmetric function $h_a$ and $h_b$ with $a\leq b$, expands with positive (integer) coefficients in the Schur basis $\{s_\mu\}_{\mu\vdash n}$ (i.e.: $\mu$ runs over the set of partitions of $n=a\,b$). For instance, we have
  $$f_{2,4}=s_{{422}}+s_{{2222}},\qquad {\rm and}\qquad
     f_{3,4}=s_{{732}}+s_{{5421}}+s_{{6222}}.$$
Here, symmetric functions are in a denumerable set of variables $\mbf{x}=x_1,x_2,x_3,\ldots$ (likewise later for $\mbf{y}=y_1,y_2,y_3,\ldots$), which are often not explicitly mentioned.
 Equivalently, Foulkes conjecture says that there is a monomorphism of $\GL(V)$-module going between the compositions of symmetric power $S^a(S^b(V))$ and $S^b(S^a(V))$, hence each $\GL(V)$-irreducible occurs with smaller multiplicity in $S^a(S^b(V))$ than it does in $S^b(S^a(V))$. Although many partial and related results have been obtained (see~\cite{briand,colmenarejo,deBoeck,manivelsylvester}), the conjecture is still open in the general case. A recent survey can be found in~\cite{cheung}. Brion~\cite{brion} has shown that it holds if $b$ is large enough with respect to $a$. Furthermore Mckay~\cite{mckay} has obtained an interesting propagation theorem which would be nice to extend to our context. Another generalization that seems to afford a natural $q$-version is that of~\cite{vessenes}. We explore this in Section-\ref{extensions}.

Our $q$-analog replaces the relevant complete homogeneous symmetric function $h_n$  by the Hall-Littlewood (or Macdonald) polynomial
    $$H_n(\mbf{x};q):=\sum_{\mu\vdash n} K_{\mu}(q) s_\mu(\mbf{x}),
    \qquad{\rm where}  \qquad K_{\lambda}(q)=\sum_\tau q^{c(\tau)},$$
  with $\tau$ running through the set of standard tableaux of shape $\lambda$, and $c(\tau)$ standing for the charge statistic. These are special cases\footnote{Because the second index is the one-part partition $(n)$.}  of Kostka-Foulkes polynomials. More details on the symmetric functions $H_n(\mbf{x};q)$ may be found in Appendix A, and especially relevant formulas are \pref{defHn} to \pref{defhnschur}. 
  
 To see how our upcoming statement corresponds to a $q$-analog of Foulkes conjecture, we recall that $H_n(\mbf{x};0)=h_n(\mbf{x})$. Hence we consider here a slightly different $q$-analog notion than the usual one, since the relevant specialization is at $q=0$ rather than $q=1$. 
\begin{conj}[$q$-Foulkes]\label{conjqFoulkes}
For any integers $0<a\leq b$, the Schur function expansion of the divided difference
\begin{equation}\label{qFoulkesconj}
  F_{a,b}(\mbf{x};q):= \frac{H_b[H_a] -H_a[H_b]}{1-q}
\end{equation}
 has coefficients in $\N[q]$, with the evident specialization $F_{a,b}(\mbf{x};0)=f_{a,b}=h_b[h_a]-h_a[h_b]$.
\end{conj}
We will explain later why it makes sense to divide by $1-q$. For instance, with $a=2$ and $b=3$, we find after calculation that  expression \pref{qFoulkesconj} expands in the Schur basis as as
\begin{eqnarray*}  F_{2,3}(\mbf{x};q)&=&
( {q}^{3}+ {q}^{2}+ q+1 )\,s_{222}+( q^2+q )\Big(( q^2+q )\,s_{{33}}+( {q}^{3}+ {q}^{2}+ q+1 )\, s_{{321}}\\
&&\qquad\qquad+(q^2+q )\, s_{{3111}} +({q}^{4}+{q}^{3}+2\,{q}^{2}+q+1)\, s_{{2211}}\\
 &&\qquad\qquad+( 2\,{q}^{3}+q^2+q )\, s_{{21111}}+
( {q}^{4}+q^2 )\, s_{{111111}}\Big).
\end{eqnarray*} 
This does specialize, at $q=0$, to the corresponding case of Foulkes conjecture:
\begin{equation}
   f_{2,3}=h_3[h_2]-h_2[h_3]=s_{222}
 \end{equation}
 Incidentally, the classical Hilbert's reciprocity law~\cite{hermite} is equivalent to the fact that all Schur functions occurring in the difference $h_b[h_a]-h_a[h_b]$ are indexed par partitions having at least $3$ parts. As seen above, this reciprocity law does not generalize to our $q$-version, since the Schur expansion of $F_{2,3}(\mbf{x};q)$ involves $s_{33}(\mbf{x})$. 
 
 A second part of Foulkes conjecture,  shown to be true by Brion~\cite{brion}, concerns the stability of coefficients as $b$ grows while $a$ remains fixed. To simplify its statement, we consider the linear operator which sends a Schur function $s_\mu(\mbf{x})$ to $s_{\overline{\mu}}(\mbf{x})$, where $\overline{\mu}$ is the partition obtained by removing the largest part in $\mu$. Let us write $\overline{f}$ for the effect of this operator on a symmetric function $f$. 
For example, we get
   $$\overline{s_{622}+s_{442}+s_{4222}+s_{22222}} =s_{22}+s_{42}+s_{222}+s_{2222},$$
  which is clearly not homogeneous.
Using this notation convention, the second part of Foulkes conjecture states that, for all $a\leq b$, the Schur expansion of
\begin{equation}\label{foulkes2}
  \overline{f_{a,b+1}}-\overline{f_{a,b}}= \overline{(h_{b+1}[h_a]-h_a[h_{b+1}])} -\overline{ (h_{b}[h_a]-h_a[h_{b}])}
\end{equation}
also affords positive integers polynomials as coefficients. 
Observe that the ``Bar'' operator allows the comparison   homogeneous functions of different degrees, namely $f_{a,b+1}$ of degree $a(b+1)$ with $f_{a,b}$ of degree $ab$.
Instances of~\pref{foulkes2} are
\begin{eqnarray*}
  \overline{f_{2,4}}- \overline{f_{2,3}}&=&\overline{(s_{422}+s_{2\bleu{222}})} -\overline{s_{222}},  \qquad (\hbox{where}\ \overline{s_{422}}\ \hbox{cancels with}\  \overline{s_{222}}),\\
              &=&s_{\bleu{222}},\\
   \overline{f_{2,5}}- \overline{f_{2,4}}&=&\overline{(s_{622}+s_{4\bleu{42}}+s_{4222}+s_{2\bleu{2222}})} -\overline{(s_{422}+s_{2222})},\\
      &=&s_{\bleu{42}}+s_{\bleu{2222}},\\
   \overline{f_{2,6}}- \overline{f_{2,5}}&=&s_{44}+s_{422}+s_{22222}.
\end{eqnarray*}
In~\cite{brion}, Brion has reduced \pref{foulkes2} to showing, which he did, that $\scalar{h_a[h_b]}{s_\lambda}\leq \scalar{h_a[h_{b+1}]}{s_{\lambda+(a)}}$,
where the \define{sum} $\lambda+\mu$ of two partitions $\lambda$ and $\mu$, is the partitions whose parts are $\lambda_i+\mu_i$ (with the convention that $\lambda_i=0$ if $i$ greater than the number of parts of $\lambda$).
A similar phenomenon also seems to hold in our context, leading us to state the following.
\begin{conj}[$q$-stability]
For any integers $0<a\leq b$, and any Schur function $s_\lambda$, we have
\begin{equation}\label{qFoulkesconjstab}
     \langle\overline{F_{a,b+1}(\mbf{x};q)}-\overline{F_{a,b}(\mbf{x};q)},s_\lambda(\mbf{x})\rangle\in \N[q].
  \end{equation}
\end{conj}
Here we use the usual scalar product $\langle -,-\rangle$ on symmetric function, for which the Schur functions form an orthonormal basis.
The smallest non-trivial example is:
\begin{eqnarray*}
\overline{F_{2,4}(\mbf{x};q)}-\overline{F_{2,3}(\mbf{x};q)}&=&(1+q)\Big(\left( \petit{{q}^{3}+2\,{q}^{4}+2\,{q}^{5}+{q}^{6}} \right) s_{{3}}+ \left(\petit{{q}^{2}+2\,{q}^{3}+3\,{q}^{4}+3\,{q}^{5}+2\,{q}^{6}+{q}^{7}} \right) s_
{{21}}\\
&& + \left(\petit{ {q}^{3}+2\,{q}^{4}+2\,{q}^{5}+{q}^{6}} \right) s_{{111}}
+ \left(\petit{ {q}^{2}+2\,{q}^{4}+{q}^{6} }\right) s_{{4}}\\
&& + \left( \petit{q+2\,{q}^{2}+6\,{q}^{3}+7\,{q}^{4}+9\,{q}^{5}+6\,{q}^{6}+4\,{q}^{7}+{q}^{8}} \right) s_{{31}}\\
&& + \left( \petit{q+3\,{q}^{2}+4\,{q}^{3}+7\,{q}^{4}+5\,{q}^{5
}+6\,{q}^{6}+2\,{q}^{7}+2\,{q}^{8} }\right) s_{{22}}\\
&& + \left( \petit{2\,{q}^{2}
+6\,{q}^{3}+10\,{q}^{4}+13\,{q}^{5}+11\,{q}^{6}+8\,{q}^{7}+3\,{q}^{8}+
{q}^{9}} \right) s_{{211}}\\
&& + \left( \petit{{q}^{3}+4\,{q}^{4}+6\,{q}^{5}+7\,{q}
^{6}+4\,{q}^{7}+2\,{q}^{8} }\right) s_{{1111}}\\
&& + \left( \petit{q+2\,{q}^{2}+5\,
{q}^{3}+6\,{q}^{4}+8\,{q}^{5}+6\,{q}^{6}+5\,{q}^{7}+2\,{q}^{8}+{q}^{9}}
 \right) s_{{32}}\\
&& + \left( \petit{3\,{q}^{2}+4\,{q}^{3}+10\,{q}^{4}+9\,{q}^{5}
+11\,{q}^{6}+6\,{q}^{7}+4\,{q}^{8}+{q}^{9} }\right) s_{{311}}\\
&& + \left( \petit{2
\,q+4\,{q}^{2}+9\,{q}^{3}+12\,{q}^{4}+15\,{q}^{5}+13\,{q}^{6}+11\,{q}^
{7}+6\,{q}^{8}+3\,{q}^{9}+{q}^{10}} \right) s_{{221}}\\
&& + \left( \petit{2\,{q}^{2
}+6\,{q}^{3}+11\,{q}^{4}+16\,{q}^{5}+17\,{q}^{6}+14\,{q}^{7}+9\,{q}^{8
}+4\,{q}^{9}+{q}^{10}} \right) s_{{2111}}\\
&& + \left( \petit{{q}^{3}+3\,{q}^{4}+6
\,{q}^{5}+7\,{q}^{6}+8\,{q}^{7}+5\,{q}^{8}+3\,{q}^{9}+{q}^{10}}
 \right) s_{{11111}}\\
&& +\rouge{  \left( \petit{1+2\,{q}^{2}+{q}^{3}+4\,{q}^{4}+2\,{q}^{5
}+5\,{q}^{6}+{q}^{7}+3\,{q}^{8}+{q}^{10}} \right) s_{{222}}}\\
&& +\left( \petit{q+{
q}^{2}+4\,{q}^{3}+5\,{q}^{4}+9\,{q}^{5}+8\,{q}^{6}+9\,{q}^{7}+5\,{q}^{
8}+4\,{q}^{9}+{q}^{10}+{q}^{11}} \right) s_{{2211}}\\
&& + \left( \petit{{q}^{2}+{q}
^{3}+5\,{q}^{4}+5\,{q}^{5}+9\,{q}^{6}+7\,{q}^{7}+7\,{q}^{8}+3\,{q}^{9}
+2\,{q}^{10}}\right) s_{{21111}}\\
&& + \left( \petit{{q}^{3}+{q}^{4}+3\,{q}^{5}+3
\,{q}^{6}+4\,{q}^{7}+3\,{q}^{8}+3\,{q}^{9}+{q}^{10}+{q}^{11}} \right) s
_{{111111}}\\
&& + \left( \petit{{q}^{4}+{q}^{6}+{q}^{8}+{q}^{10}} \right) s_{{
1111111}}\Big),
\end{eqnarray*}
and we do observe that this specializes to the (much simpler) classical Foulkes case when we set $q=0$. Another stability in the vein of Manivel~\cite{manivelgaussian},  that seems to hold in our $q$-context, is that
   $$(\overline{F_{a+1,b+1}(\mbf{x};q)}-\overline{F_{a+1,b}(\mbf{x};q)})-
      (\overline{F_{a,b+1}(\mbf{x};q)}-\overline{F_{a,b}(\mbf{x};q)}),$$
is Schur positive, for all $a<b$.

\section{Supporting facts and results}\label{resultats}
  Beside having checked our conjectures by explicit computer calculation for all $1<a\,b\leq 25$, we now show that they hold true at $q=1$. Moreover we give a few interesting related results.
First off,  division by $1-q$ makes sense in the statement of \pref{qFoulkesconj}, since for any $a$ and $b$, we have $H_a[H_b](\mbf{x};1)= H_b[H_a](\mbf{x};1)$.
   This is almost immediate, since $H_n(\mbf{x};1)=p_1^n$ and  evaluation at $q=1$ is compatible\footnote{Meaning that they commute as operators. Observe that this is not so with the evaluation at $q=-1$.} with plethysm, so that $H_a[H_b](\mbf{x};1)=p_1^{ab}$. 
Hence, the $q$-polynomial $F_{a,b}(\mbf{x};q)$ (with coefficients in $\Lambda$) vanishes at $q=1$, and it is divisible by $1-q$.   
  
\subsection*{Dimension Count}
As mentioned previously, when a homogeneous degree $n$ symmetric function $f$ occurs as a (graded) Frobenius transform of the character of an $\S_n$-module, the dimension (Hilbert series) of this module may be readily calculated
by taking its scalar product with $p_1^n$. On the other hand, general principles insure that there exists such a module (albeit not explicitly known) whenever $f$ expands positively (with coefficients in $\N[q]$) in the Schur function basis. Finding an explicit formula for this ``dimension'' may give a clue on what kind of module one should look for in order to prove the conjectures.
With this in mind, let us set the notation 
   $$\dim(f):=\langle p_1^n,f\rangle.$$
For instance, we may easily calculate that
\begin{equation}\label{partionab}
    \dim(h_a[h_b]) = \frac{(ab)!}{a!\, b!^a},
 \end{equation}
since $p_1^{ab}$ may only occur in the plethysm $h_a[h_b]$ as
     $$\frac{p_1^a}{a!}\left[\frac{p_1^b}{b!}\right]= \frac{p_1^{ab}}{a!\, b!^a}. $$
In a classical combinatorial setup, formula~\pref{partionab} is easily interpreted as the number of partitions of a set of cardinality $ab$, into blocks each having size $b$. We say that this is a $b^a$-partition.
Indeed, using a general framework such as the Theory of Species (see~\cite{species}), it is well understood that $h_a[h_b]$ may be  interpreted as the Polya cycle index enumerator of such partitions, i.e.:
    $$h_a[h_b]  = \frac{1}{n!}\sum_{\sigma\in \S_{n}}
            \fix_\sigma p_1^{d_1}p_2^{d_2}\cdots p_n^{d_n},$$
   where $n=ab$, and $d_k$ denotes the number of cycles of size $k$ in $\sigma$. Here, we further denote by $ \fix_\sigma$ the number of $b^a$-partitions that are fixed by a permutation $\sigma$, of the underlying elements. It follows that 
\begin{equation}\label{dimFoulkes}
    \dim(F_{a,b}(\mbf{x};0)) = (ab)!\left(\frac{1}{b!\, a!^b}-  \frac{1}{a!\, b!^a}\right),
  \end{equation}
  is the difference between the number of $a^b$-partitions and $b^a$-partitions. Some authors have attempted to exploit this fact to prove Foulkes conjecture (for positive and negative results along these lines see~\cite{paget,pylyavskyy,sivek}).
  
 It is interesting that we have the following very nice $q$-analog (at $0$) of \pref{dimFoulkes}.
 \begin{prop} For all $a<b$, we have
\begin{equation}\label{dimqFoulkes}
    \dim(F_{a,b}(\mbf{x};q)) = \frac{(ab)!}{1-q}\left(\frac{[b]_q!}{b!}\frac{([a]_q!)^b}{a!^b}-\frac{[a]_q!}{a!}\frac{([b]_q!)^a}{b!^a}\right),
  \end{equation}
and, letting $q\mapsto 1$, we find that
  \begin{equation}\label{dim1Foulkes}
    \dim(F_{a,b}(\mbf{x};1)) = \frac{(ab)!(a-1)(b-1)(b-a)}{4}.
  \end{equation}
 \end{prop}
  \begin{proof}[\bf Proof]
We first calculate $\dim(H_a[H_b])$ 
directly as follows

Now, exploiting classical properties of the logarithmic derivative $\mathrm{D}_{\mathrm{log}}\, f:=f'/f$ (with respect to $q$), we easily calculate that 
    $$\mathrm{D}_{\mathrm{log}} \,[b]_q!\,([a]_q!)^b  \Big|_{q=1}= \frac{1}{2}\binom{b-1}{2}+\frac{b}{2}\binom{a-1}{2}.
    $$
From this we may readily obtain that $\lim_{q\to 1}  \dim(F_{a,b}(\mbf{x};q)) $ gives~\pref{dim1Foulkes}.
 \end{proof}

\newcommand{\sbinom}[2]{\Big(\genfrac{}{}{0pt}{0}{#1}{#2}\Big)}
 \section{The conjecture holds at \texorpdfstring{$q=1$}{q}}
We start with an explicit formula that will be helpful in the sequel, setting the simplifying notation
\begin{equation}
    \mathcal{E}_b:=( \left( h_{{2}}+e_{{2}} \right) ^{b}+ \left( h_{{2}}-e_{{2}} \right) ^{b} )/2,\qquad{\rm and}\qquad
    \mathcal{O}_b:=\frac{1}{2}\,\left( ( h_{{2}}+e_{{2}} ) ^{b}- ( h_{{2}}-e_{{2}} ) ^{b} \right),
 \end{equation}
 for the odd-part of $( h_{{2}}+e_{{2}} ) ^{b}$ in the ``variable'' $e_2$.

\begin{lemma} For all $a$ and $b$, we have the divided difference evaluation
   \begin{equation}\label{formuledelH}
     \lim_{q\to 1}  \frac{h_1^{ab} -H_a[H_b]}{1-q}= 
    a\, \sbinom{b}{2}\,h_1^{ab-2}\,e_2+\sbinom{a}{2}\, h_1^{(a-2)b}\, \mathcal{O}_b.
    \end{equation}
\end{lemma}
\begin{proof}[\bf Proof]
    This is a straightforward calculation using rules~\pref{plethrules} and~\pref{defhnpower}.   \end{proof}
It immediately follows that we have the following formula.
 \begin{prop}
   For any $a>1$ and $b>a$, we have
\begin{equation}
   F_{a,b}(\mbf{x};1)=
   \frac{1}{2}\left(ab\,( b-a) \,h_1^{ab-2}\,e_2
   +a \left( a-1\right) h_1^{(a-2)b}\,\mathcal{O}_b
   -b \left( b-1 \right) h_1^{a(b-2)}\,\mathcal{O}_a\right).
\end{equation}
Moreover, this a positive integer coefficient polynomial in $h_1$, $h_2$ and $e_2$; hence, it expands positively in the Schur basis.
\end{prop}
 We also have the following recursive approach to the calculation of $F_{a,b}$, as a polynomial in $h_1$, $h_2$ and $e_2$.
\begin{prop}\label{proprec}
 \begin{equation}
   F_{a,b+1}(\mbf{x};1)=h_1^a\, F_{a,b}(\mbf{x};1)+    
                                         2\,h_1^{(a-2) b}\,\Theta_a(b),
\end{equation}
with $ \Theta_a(b)$ satisfying the recurrence 
$$ \Theta_a(b) = (3\,h_2+e_2)\,\Theta_a(b-1)-h_1^2\,(3\,h_2-e_2)\,\Theta_a(b-2)+h_1^4\,(h_2-e_2)\,\Theta_a(b-3),$$
with initial conditions: $\Theta_a(a) = \Theta_{a-1}(a)$, $\Theta_a(a+1) =({a^2e_2}/{2})\, \mathcal{E}_a+({a\,h_2}/{2})\,\mathcal{O}_a$, and
$$ \Theta_a(a+2) = 
\frac{(  \left( a+1 \right) ^{2}-2 )\,e_2}{2} \, \mathcal{E}_{{a+1}}
-\frac{(a+1)\,h_2}{2}\,\mathcal{O}_{{a+1}}+e_2\, ( a\,e_{{2}}+h_{{2}} )\, ( h_2-e_2 )^a.
$$
Moreover, $\Theta_a(b)$ lies in $\N[h_1,h_2,e_2]$.
 \end{prop}

\noindent
Using one of these calculation techniques, we find that
$$\begin{array}{lll}
F_{{2,3}}=4\,e_2^{3}, &
F_{{2,4}} =8\,e_2^{3}\, ( e_2+2\,h_2 )\\[6pt]& 
F_{{2,5}} =8\,e_2^{3}\,( 2\,e_2^{2}+5\,h_2\,e_2+5\,h_2^{2} ), \\[8pt]
F_{{3,4}}=24\,h_1^{4}\,e_2^{3}\,h_2 ,&
F_{{3,5}}=8\,h_1^{5}\,e_2^{3}\, ( e_2^{2}+5\,h_2\,e_2+10\,h_2^{2}), \\[6pt]&
F_{{3,6}}=12\,h_1^{6}\,e_2^{3}\, ( e_2^{3}+9\,e_2^{2}\,h_2+15\,e_2\,h_2^{2}+15\,h_2^{3} ), \\[8pt]
F_{{4,5}}=16\,h_1^{10}\,e_2^{3}\, ( e_2^{2}+5\,h_2^{2} ), &
F_{{4,6}}=24\,h_1^{12}\,e_2^{3}\, ( e_2^{3}+4\,e_2^{2}\,h_2+5\,e_2\,h_2^{2}+10\,h_2^{3}), \\[6pt]&
F_{{4,7}}=24\,h_1^{14}\,e_2^{3}\, ( 2\,e_2^{4}+7\,e_2^{3}\,h_2+21\,e_2^{2}\,h_2^{2}+21\,e_2\,h_2^{3}+21\,h_2^{4} ) .
\end{array} $$
In particular, for all $b>a>1$, we have.
   $$F_{a,b+1}(\mbf{x};1)-h_1^a F_{a,b}(\mbf{x};1)\in \N[h_1,h_2,e_2],$$
which implies the analog at $q=1$ of the stability portion of Foulkes conjecture, namely
\begin{prop}
    For all $a<b$, and all partition $\lambda$, we have
         $$  \langle\overline{F_{a,b+1}(\mbf{x};1)}-\overline{F_{a,b}(\mbf{x};1)},s_\lambda(\mbf{x})\rangle\in \N.$$
\end{prop} 
\begin{proof}[\bf Proof]
 Indeed, using the classical Pieri rule for the calculation of $h_1\, s_\lambda$, it is easy to see that
     $$\overline{h_1^a F_{a,b}(\mbf{x};1))}-\overline{F_{a,b}(\mbf{x};1)}$$
   is Schur positive, since one of the terms in $h_1^a s_\lambda$ is the Schur function indexed by the partition obtained from $\lambda$ by adding $a$ boxes to its first line. Hence the lemma directly implies that
      $$\overline{F_{a,b+1}(\mbf{x};1)}-\overline{F_{a,b}(\mbf{x};1)}=
          (\overline{F_{a,b+1}(\mbf{x};1)}-\overline{h_1^a F_{a,b}(\mbf{x};1)})+
          (\overline{h_1^a F_{a,b}(\mbf{x};1))}-\overline{F_{a,b}(\mbf{x};1)})$$
   is  Schur positive.
\end{proof}
It is interesting to calculate how $F_{a,b}(\mbf{x};q)$ expands explicitly as a polynomial in $q$. Indeed, by a direct calculation, one gets
    $$F_{a,b}(\mbf{x};q) = (h_b[h_a]-h_a[h_b])+
          (h_{b-1}[h_a]\,h_{a-1}\,h_1 - h_{a-1}[h_b]\,h_{b-1}\,h_1)q+\ldots$$
 with other more intricate terms. Hence, the conjectured Schur-positivity of $F_{a,b}(\mbf{x};q)$ implies that we have Schur positivity of
      $$ ((h_{b-1}[h_a])\cdot h_{a-1} -( h_{a-1}[h_b])\cdot h_{b-1})\cdot h_1.$$
Experiments suggest that we  in fact have Schur positivity of    
\begin{equation}
    (h_{b-1}[h_a])\cdot h_{a-1} -( h_{a-1}[h_b])\cdot h_{b-1},
 \end{equation}
 for all $a<b$,  which would immediately imply the above.
                                       
\section{A first extension}\label{extensions}
In her thesis, partly presented in~\cite{vessenes},  Vessenes attributes\footnote{However, it is not clear in the first paper cited where this exact statement can be found in~\cite{doran}.} the following generalization of Foulkes conjecture to Doran~\cite{doran}, and a similar extension is considered (and proved in a special case) in~\cite{Abdesselam}.  
 For $a$ and $b$, let $c$ be a divisor of $n:=ab$, lying between $a$ and $b$, and set $d=n/c$.
 Then the generalized conjecture of Doran states that
    \begin{equation}\label{doran}
         h_c[h_d]-h_a[h_b]\in \N[s_\mu\ |\ \mu\vdash n].
    \end{equation}
This is to say that this difference is Schur positive. 
For example, one calculates that
    $$h_3[h_4]-h_2[h_6]=s_{{93}}+s_{{444}}+s_{{642}}+s_{{741}}+s_{{822}}.$$
Our experiments suggest that this extends to our $q$-context, in a manner that is compatible with our previous discussion.
\begin{conj}\label{conjqDoran} Let $c$ be a divisor of $n:=ab$, with $a\leq c\leq b$, and set $d=n/c$. Then the divided difference
\begin{equation}\label{conjqDoranformule}
   0\preceq_s \frac{H_c[H_d]-H_a[H_b]}{1-q}.
  \end{equation}
In other words, its expansion in the Schur basis has polynomial coefficients in $q$, with coefficients in $\N$.
\end{conj}
Clearly, $c$ may be chosen to be equal to $b$, thus getting our previous Conjecture-\ref{conjqFoulkes}.
Once again, using Formula~\pref{formuledelH}, we can calculate how this specializes at $q=1$. Setting $n:=ab=cd$, we get
\begin{eqnarray}
   \lim_{q\to1} \frac{H_c[H_d]-H_a[H_b]}{1-q}&=& \nonumber\\
    &&\hskip-.5in \frac{1}{2}\left(n\,( b-d) \,h_1^{n-2}\,e_2
   +a\, ( a-1)\, h_1^{n-2b}\,\mathcal{O}_b
   -c\, ( c-1 )\, h_1^{n-2d}\,\mathcal{O}_d\right).
\end{eqnarray} 
and again the resulting symmetric function lies in $\N[h_1,h_2,e_2]$.
For instance, we have 
    $$\lim_{q\to 1}\frac{H_3[H_4]-H_2[H_6]}{1-q}\,=
    2\,e_{{2}} \left( 6\,e_2^{5}+27\,e_2^{4}h_{{2}}+48\,e_2
^{3}h_2^{2}+58\,e_2^{2}h_2^{3}+18\,e_{{2}}h_2^
{4}+3\,h_2^{5} \right).$$

\section{Expanding Foulkes conjecture to more general diagrams}
 For partitions $\alpha$, $\beta$, $\gamma$, and $\delta$, none of which equal to $(1)$ and such that $|\alpha|\cdot |\beta|=|\gamma|\cdot |\delta|=n$, let us say that $\foulkes{\alpha}{\beta}{\gamma}{\delta}$ is a \define{Foulkes configuration} for $n$,  if and only if
\begin{equation}\label{Fconf}
   s_\alpha[s_\beta]\  \prec_s\  s_\gamma[s_\delta].
\end{equation}
Observe that, since we ask for strict inequality, we are excluding the trivial case $(\alpha,\beta)=(\gamma,\delta)$.
Clearly, for $1<a<b$, Foulkes conjecture says that $\foulkes{a}{b}{b}{a}$ is a Foulkes configuration. Likewise statement~\pref{doran}, under the conditions there specified, is equivalent to saying that $\foulkes{a}{b}{c}{d}$ is a Foulkes configuration. Clearly, there are no Foulkes configurations for $n$ prime, and one may check by direct calculations that there are none for $n=4$ and $n=9$.
Again by direct explicit calculation of all possible cases, we get that the only Foulkes configurations for $n=6$ (the smallest non-trivial situation) are
  $$\foulkes2332,\qquad
\foulkes{11}{111}{3}{11},\qquad
\foulkes{111}{2}{11}{21},\qquad{\rm and}\qquad
\foulkes{111}{11}{2}{21};
$$
for $n=8$,  the $14$ configurations:
$$\begin{array}{lllll}
\foulkes{2}{4}{4}{2},&
\foulkes{2}{1111}{4}{11},&
\foulkes{11}{4}{31}{2},&
\foulkes{11}{22}{31}{2},\\[6pt]
\foulkes{11}{22}{31}{11},&
\foulkes{11}{31}{211}{2},&
\foulkes{11}{211}{211}{11},&
\foulkes{11}{1111}{31}{11},\\[6pt]
\foulkes{22}{2}{2}{31},&
\foulkes{22}{11}{2}{211},&
\foulkes{211}{2}{11}{31},&
\foulkes{211}{11}{11}{211},\\[6pt]
\foulkes{1111}{2}{2}{31},&
\foulkes{1111}{11}{2}{211};
\end{array}$$
and for $n=10$, the $8$ configurations:
$$\begin{array}{lllll}
\foulkes{2}{5}{5}{2},&
\foulkes{2}{221}{311}{11},&
\foulkes{2}{2111}{311}{11},&
\foulkes{11}{32}{311}{2},\\[6pt]
\foulkes{11}{41}{311}{2},&
\foulkes{11}{11111}{5}{11},&
\foulkes{11111}{2}{2}{311},&
\foulkes{11111}{11}{11}{311}.
\end{array}$$
For  $n$ up to $16$, Table-\ref{table_foulkes} gives the number of Foulkes configurations for $n$.
\begin{table}[ht]
$$\begin {array}{|c|c|c|c|c|c|c|c|c|c|c|c|c|c|c|c|c|} 
\hline
n & 1 & 2 & 3 & 4 & 5 & 6 & 7 & 8 & 9 &10 &11 &12 &13 &14 &15&16\\
\hline
\#\ {\rm Config}& 0& 0& 0& 0&0&4&0& 14& 0&8& 0& 110&0&24&17&221\\
\hline
\end {array}$$
\caption{Number of Foulkes configurations for $n$.}\label{table_foulkes}
\end{table}
Thus, there seems to exist an abundance of such. As we will see, the picture becomes much sharper when we consider the $q$-setup.

Let us now consider the following\footnote{This is well-known in the theory of Macdonald polynomials, and all properties also mentioned.} $q$-analog of Schur functions:
    $$S_\mu(\mbf{x};q):=\omega\, q^{n(\mu')}\,H_\mu(\mbf{x};1/q,0),$$
defined in terms of specialization at $t=0$ of the combinatorial Macdonald polynomials\footnote{See Appendix A for various notations used here.} $H_\mu(\mbf{x};q,t)$, with $\omega$ standing for the ``usual'' linear involution that sends $s_\mu(\mbf{x})$ to $s_{\mu'}(\mbf{x})$.
 In particular, one may check that $S_\mu(\mbf{x};0)=s_\mu(\mbf{x})$ and $S_\mu(\mbf{x};1)=e_{\mu'}(\mbf{x})$. 
For instance, we have
      $$S_{32}(\mbf{x};q)=s_{{32}}(\mbf{x})+q\,s_{{311}}+q \left( q+1 \right) s_{{221}}(\mbf{x})+{q}^{2} \left( q+1
 \right) s_{{2111}}(\mbf{x})+{q}^{4}s_{{11111}}(\mbf{x}),$$
 and $S_{32}(\mbf{x};1)=e_1(\mbf{x})\,e_2(\mbf{x})^2$.
 Observe that all terms in the Schur expansion of $S_\mu(\mbf{x};q):$ are indexed by partitions that are dominated by $\mu$.
Moreover we get back our previous context for $\mu=(a)$, since $S_{(a)}(\mbf{x};q)=H_a(\mbf{x},q)$.
Under the same assumptions as in~\pref{Fconf} for the partitions involved, we say the we have a \define{$q$-Foulkes configuration} denoted $\qfoulkes{\alpha}{\beta}{\gamma}{\delta}$, if and only if 
 \begin{equation}\label{qfoulkesconfig}
    0\prec_s \frac{S_\gamma[S_\delta]-S_\alpha[S_\beta]}{1-q},
 \end{equation}
with the right-hand side having polynomial coefficients in $q$. 
In particular, this last condition requires that, at $q=1$ we have the equality
   $$ S_\alpha[S_\beta]\Big|_{q=1} = S_\gamma[S_\delta]\Big|_{q=1},$$
 which is equivalent to
\begin{equation}\label{teste}
   e_{\alpha'}[e_{\beta'}] = e_{\gamma'}[e_{\delta'}] .
\end{equation}
 For instance, it is easy to check that this last equality holds when
 \begin{equation}\label{bonnecond}
   \alpha=a,\qquad \beta=\underbrace{bb\cdots b}_k,\qquad \gamma=c,\qquad {\rm and}\qquad
        \delta=\underbrace{dd\cdots d}_k,
   \end{equation}
for any $a,b,c,d,k$ in $\N$, such that $ab=cd$, since both sides of~\pref{teste} evaluate to $e_k^{a+b}$. Evidently, all $q$-Foulkes configurations are also Foulkes configurations, but most Foulkes configurations do not satisfy the extra requirement that~\pref{bonnecond} holds. Explicit calculations reveal that this condition significantly reduces the number of possibilities.

\section{Experimental data for the \texorpdfstring{$q$}{q}-diagram expansion} For $n<12$, the only $q$-Foulkes configurations are those that correspond to Conjecture-\ref{conjqDoran}. For $n=12$, on top of the cases considered in the conjecture in question, there is but one extra $q$-Foulkes configuration, which is $\qfoulkes3{22}2{33}$. For $n=14$ and $n=15$, there is no extra $q$-Foulkes configuration, beside those predicted by Conjecture-\ref{conjqDoran}. For $n=16$, we have the $q$-Foulkes configurations:
    $$\qfoulkes2882,\qquad
        \qfoulkes2844,\qquad
       \qfoulkes2{44}4{22},$$
for $n=18$:
$$\begin{array}{lll}
\qfoulkes2936, \quad& \qfoulkes2963, \quad& \qfoulkes2992, \\[6pt]
\qfoulkes3663,  & \qfoulkes2{333}3{222}, &\qfoulkes2{63}3{42};
\end{array}$$
and for $n=20$:
$$\begin{array}{lll}
\qfoulkes2{55}5{22},\quad& \qfoulkes2{10}45,\quad& \qfoulkes2{10}54, \\[6pt]
\qfoulkes2{10}{10}2, & \qfoulkes4554;
\end{array}$$
hence we get the count of Table-\ref{table_qfoulkes}.
\begin{table}[ht]
$$\begin {array}{|c|c|c|c|c|c|c|c|c|c|c|c|c|c|c|c|c|c|c|c|c|} 
\hline
n & 1 & 2 & 3 & 4 & 5 & 6 & 7 & 8 & 9 &10 &11 &12 &13 &14 &15&16&17&18&19&20\\
\hline
\#\ q\hbox{-Config}& 0& 0& 0& 0&0&1&0& 1& 0&1& 0& 5&0&1&1&3&0&6&0&5\\
\hline
\end {array}$$
\caption{Number of $q$-Foulkes configurations for $n$.}\label{table_qfoulkes}
\end{table}
In view of this, and the fact that~\pref{teste} holds when we have~\pref{bonnecond}, it is tempting to ``guess'' that for any $n=abk=cdk$,  with $2\leq a<c\leq b$, we should have the extra $q$-Foulkes configuration 
\begin{equation}\label{Foulkesabbb}
   \qfoulkes{a}{\underbrace{bb\cdots b}_k}{c}{\underbrace{dd\cdots d}_k},
 \end{equation}
thus explaining almost all cases up to $n=20$. Moreover, this ``guess'' does check out with the configurations:
$$\begin{array}{lll}
   \qfoulkes2{66}6{22},\qquad & \qfoulkes2{66}3{44},\qquad& \qfoulkes2{66}4{33},\\[6pt]
   \qfoulkes3{44}4{33},& \qfoulkes3{55}5{33},&\qfoulkes2{444}4{222},\\[6pt]
   \qfoulkes2{555}5{222},& \qfoulkes2{3333}3{2222},&  \qfoulkes2{33333}3{22222}.
   \end{array}$$
Formulate explicily, in the $q=0$ context, the Schur positivity associated to the configurations~\pref{Foulkesabbb} would correspond to the following nice extension of conjecture~\pref{doran}
   \begin{equation}
       h_c[s_{dd\cdots d}]\preceq_s h_b[s_{aa\cdots a}],
   \end{equation} 
with the same conventions therein, and with $k$ repeated parts in both instances. This led us to consider the similarly constituted inequality
   \begin{equation}
       h_a[h_b^k]\preceq_s h_c[h_d^k],
   \end{equation} 
 which also seems to check out experimentally.
 
Other somewhat similarly flavored $q$-Foulkes configurations are:  
$$\begin{array}{lll}
   \qfoulkes2{63}3{42},\qquad & \qfoulkes2{84}4{42},\qquad & \qfoulkes2{93}3{62},\\
   \qfoulkes2{96}3{64}, &  \qfoulkes2{663}3{442};
   \end{array}$$
while the following ones, who are very close in structure to those above, are not Foulkes configurations (so that they cannot be $q$-Foulkes configurations either):
 $$ [(2, 633), (3, 422)],\qquad   
  [(2, 933), (3, 622)],\qquad {\rm and}\qquad 
   [(2, 6333), (3, 4222)].$$
We have also found that there are even more intricate $q$-Foulkes configurations such as: 
   $$\qfoulkes2{(10, 4)}4{52},\qquad  \qfoulkes2{(10, 5)}5{42},\quad{\rm or}\quad \qfoulkes2{(12, 3)}3{82}. $$ 
Hence, a more extensive exploration is certainly needed here if we wish to explicitly characterize all possible cases. 

\section{From diagram Foulkes to diagram \texorpdfstring{$q$}{q}-Foulkes}
An intriguing development, explicitly checked out for all cases\footnote{Involving $67$ configurations in total.} with $n$ up to $30$, is that having both the necessary conditions~\pref{Fconf} and~\pref{teste} holding seems to be equivalent to having the full $q$-Schur positivity~\pref{qfoulkesconfig} holding too. In other words, we have the following general statement, which would reduce all $q$-versions to the $q=0$ case.

\begin{conj}\label{conjstrong}
 For partitions $\alpha$, $\beta$, $\gamma$, and $\delta$, such that $e_{\alpha'}[e_{\beta'}]=e_{\gamma'}[e_{\delta'}]$, we have 
\begin{equation}\label{suffcond}
  s_\alpha[s_\beta]\  \prec_s\  s_\gamma[s_\delta],\qquad \hbox{if and only if} \qquad  0\prec_s \frac{S_\gamma[S_\delta]-S_\alpha[S_\beta]}{1-q}.
\end{equation}
 \end{conj}
Clearly, when both $\alpha$ and $\gamma$ are one part partitions, respectively equal to $a$ and $c$, the second condition in~\pref{suffcond} is simply that $e_{\beta'}^a=e_{\delta'}^c$. Only this simpler version of the second condition is needed (together with the first one) in all cases explicitly calculated as detailed above. In other words, all configurations that we have found to satisfy \pref{suffcond} are such that $\alpha$ and $\gamma$ are reduced to one part.

It is worth noticing that, when both $\beta$ and $\delta$ are also one part partitions, the above conjecture says that \pref{doran} implies Conjecture-\ref{conjqDoran}; since the second condition is trivially verified in those instances. In view of earlier comments, the classical Foulkes conjecture would also imply its $q$-analog. Hence, showing Conjecture-\ref{conjstrong} would neatly wrap up the $q$-story.

\section{Iterated plethysm generalizations} Another intriguing possible extension\footnote{Which we don't know yet how to extend correctly to the $q$-context.} in the classical context consists in considering signed-combinations of higher iterated plethystic compositions, such as
\begin{equation}\label{trois}
   h{\langle c,b,a\rangle} -h{\langle b,c,a\rangle}-h{\langle c,a,b\rangle}+h{\langle b,a,c\rangle}+h{\langle a,c,b\rangle}-h{\langle a,b,c\rangle},
 \end{equation}
  for $a<b<c$, where we use the notation
     $$h{\langle a_1,a_2,\ldots,a_n\rangle}:=\begin{cases}
     h_{a_1}, & \text{if}\ n=1, \\[6pt]
      h_{a_1}[h{\langle a_2,\ldots,a_n\rangle}], & \text{if}\ n>1.
\end{cases}$$
for iterated plethysm to make our expressions more readable. Thus,
     $$h{\langle a,b,c\rangle}=h_a[h_b[h_c]]].$$
 For sure, the Schur positivity of certain linear combinations of iterates of plethysm follow immediately from general positivity properties of plethysm (see Appendix B). For instance, we may readily deduce from a special case of Foulkes conjecture such as $0\prec_s (h_c[h_b] -h_b[h_c])$ that
\begin{eqnarray*}   
        0\prec_s (h{\langle c,b,a\rangle} -h{\langle b,c,a\rangle}), &\quad{\rm and}\quad& 0\prec_s (h{\langle a,c,b\rangle} -h{\langle a,b,c\rangle}).
  \end{eqnarray*}
However, this is not enough to allow us to deduce~\pref{trois} from Foulkes conjecture.  

More generally, we could consider alternating sum analogs of Foulkes statement, over the symmetric group $\S_n$, of the form
     \begin{equation}\label{higher}
        0\preceq_s  \sum_{\sigma\in \S_n} \sign(\sigma)\  h{\langle a_{\sigma(1)},a_{\sigma(2)},\ldots,a_{\sigma(n)}\rangle},
     \end{equation}
when $a_1>a_2>\ldots >a_n>1$. For $n=2$, this is clearly Foulkes conjecture. We have checked for a limited number of cases, and only with $n=3$, that the resulting symmetric functions are indeed Schur positive. We underline that the degree of the symmetric functions involved is $a_1a_2\cdots a_n$, which is at least $(n+1)!$ under the hypothesis considered. Hence the verification of the Schur positivity of \pref{higher} rapidly goes beyond our computing capacity. Thus, we may not as ``safely'' as before state a conjecture to the effect that the right-hand side of~\pref{higher} should always be Schur positive.  C.~Reutenauer has also suggested that we consider ``immanant analogs'' of \pref{higher}. The simplest case, besides trivial situations or those already considered, corresponds to
    $$0\preceq_s 2h{\langle c,b,a\rangle} -h{\langle b,a,c\rangle} - h{\langle a,c,b\rangle},$$
 again for $a<b<c$.
Once more, only a rather small set of experiments suggests that Schur positivity may also hold in such a context.

\section*{Appendix A: Background on symmetric functions} \label{AppendixA}
Trying to make this text self-contained, we rapidly recall most of the necessary background on symmetric functions. As is usual, we often write symmetric functions without explicit mention of the variables. Thus, we denote by $p_k$ (as in~\cite{macdonald}) the power-sum symmetric functions
   $$p_k=p_k(x_1,x_2,x_3,\ldots) :=x_1^k+x_2^k+x_3^k+\ldots,$$ 
using which, we can expand the complete homogeneous symmetric functions as 
\begin{equation}\label{hpexpansion}
    h_n=\sum_{\mu\vdash n} \frac{p_\mu}{z_\mu},\qquad {\rm with}\qquad
       p_\mu:=p_1^{d_1}p_2^{d_2}\cdots p_n^{d_n},
   \end{equation}
where $d_k=d_k(\mu)$ is the multiplicity of the part $k$ in the partition $\mu$ of $n$, and $z_\mu$ stands for the integer
      $$z_\mu:=\prod j^{d_j}\,d_j!.$$
 For instance, we have the very classical expansions
     $$h_2=\frac{p_1^2}{2}+\frac{p_2}{2}, \qquad 
         h_3=\frac{p_1^3}{6}+\frac{p_1p_2}{2}+\frac{p_3}{3}.$$
As is also very well known, the homogeneous degree $n$ component $\lambda_n$ of the graded ring $\Lambda$ of symmetric functions, affords as a linear basis the set of Schur functions $\{s_\mu\}_{\mu\vdash n}$, indexed by partitions of $n$. Among the manifold interesting formulas regarding these, we will need the Cauchy-kernel identity.
\begin{eqnarray}
   h_n(\mbf{x}\mbf{y})&=& \sum_{\mu\vdash n} s_\mu(\mbf{x})s_\mu(\mbf{y})\label{cauchys}\\
        &=& \sum_{\mu\vdash n} \frac{p_\mu(\mbf{x})p_\mu(\mbf{y})}{z_\mu},\label{cauchyp}
    \end{eqnarray}
 with $h_n(\mbf{x}\mbf{y})=h_n(...,x_iy_j,...)$ corresponding to the evaluation of $h_n$ in the ``variables'' $x_iy_j$. Otherwise stated, we
 may express this by the generating function identity
    $$\sum_{n\geq 0} h_n(\mbf{x}\mbf{y})\,z^n=\prod_{i,j}\frac{1}{1-x_iy_j\,z}.$$ 
Plethysm is characterized by the following properties. Let $f_1$, $f_2$, $g_1$ and $g_2$ be any symmetric functions, and $\alpha$ and $\beta$ be in $\Rat$, then
 \begin{equation}\label{plethrules}
 \begin{array}{rcl}
a)& &(\alpha\,f_1+\beta\,f_2)[g]=\alpha\,f_1[g]+\beta\,f_2[g],\hskip2.5in\\[4pt]
b)& &(f_1\cdot f_2)[g]=f_1[g]\cdot f_2[g],\\[4pt]
c)& &p_k[\alpha\,g_1+\beta\,g_2]=\alpha\,p_k[g_1]+ \beta\,p_k[g_2],\\[4pt]
d)& &p_k[g_1\cdot g_2]=p_k[g_1]\cdot p_k[g_2],\\[4pt]
e)& &p_k[p_j]=p_{kj},\qquad{\rm and}\qquad p_k[q]=q^k.
 \end{array}
 \end{equation}
The first four properties reduce any calculation of plethysm to instances of the fifth one. In this context, it is useful to consider  variable sets as sums $\mbf{x}=x_1+x_2+x_3+\ldots$, so that $f[\mbf{x}]$ corresponds to the evaluation of the symmetric function $f$ in the variables $\mbf{x}$. In particular, Cauchy's formula gives an explicit expression for the expansion of 
             $$h_n[\mbf{x}\mbf{y}]=h_n[(x_1+x_2+x_3+\ldots)(y_1+y_2+y_3+\ldots)].$$
Likewise $f[1/(1-q)]=f[1+q+q^2+\ldots]$, corresponds to the evaluation $f(1,q,q^2,\ldots)$.
With all this at hand, the polynomials $H_n(\mbf{x};q)$ can be explicitly defined as
   \begin{equation}\label{defHn}
       H_n(\mbf{x};q):=[n]_q!\, (1-q)^n\,h_n\left[\frac{\mbf{x}}{1-q}\right]
   \end{equation}
where $[n]_q!$  stands for classical the $q$-analog of $n!$:
   $$[n]_q!:=[1]_q\,[2]_q\cdots [n]_q,\qquad {\rm with} \qquad
      [k]_q=1+q+\ldots+q^{k-1}.$$
Calculating with the plethystic rules~\pref{plethrules}, and formula~\pref{hpexpansion}, we get the explicit power-sum expansion
   \begin{equation}\label{defhnpower}
       H_n(\mbf{x};q) =\sum_{\mu\vdash n}\frac{[n]_q!\,}
           {z_\mu\,[\mu_1]_q[\mu_2]_q\cdots [\mu_\ell]_q}\, (1-q)^{n-\ell(\mu)}\ p_\mu(\mbf{x}),
   \end{equation}
where $\ell=\ell(\mu)$ is the number of parts of $\mu$.  
To get a Schur expansion for $H_n(\mbf{x};q)$, we recall the hook length expression
\begin{eqnarray} 
    s_\mu[1/(1-q)] &=&s_\mu(1,q,q^2,q^3,\ldots)\nonumber\\
         &=& q^{n(\mu)}\prod_{1\leq i\leq \mu_j} \frac{1}{1-q^{h_{ij}}},
    \end{eqnarray}
 where  $h_{ij}=h_{ij}(\mu)$ is the hook length of a cell $(i,j)$ of the Ferrers diagram of $\mu$, and $n(\mu):=\sum_{(i,j)} j$. Now, using Cauchy's formula~\pref{cauchys}, with $\mbf{y}=1+q+q^2+\ldots$, 
we find that
   \begin{equation}\label{defhnschur}
       H_n(\mbf{x};q) =\sum_{\mu\vdash n} \frac{q^{n(\mu)}[n]_q!}{\prod_{1\leq i\leq \mu_j}[h_{ij}]_q}\ s_\mu(\mbf{x}).
   \end{equation}
It is well known that the coefficient of $s_\mu(\mbf{x})$ occurring here is a positive integer polynomial that $q$-enumerates standard tableaux with respect to the charge statistic. This is the $q$-hook formula. Thus, we find the two expansions.
\begin{eqnarray*}
     H_3(\mbf{x};q)&=& \frac{[2]_q[3]_q}{6}\,p_1(\mbf{x})^3+\frac{[3]_q}{2}(1-q)\,p_1(\mbf{x})p_2(\mbf{x})+\frac{[2]_q}{3}(1-q)^2\,p_3(\mbf{x}),\\[4pt]
           &=& s_3(\mbf{x})+(q+q^2)s_{21}(\mbf{x})+q^3s_{111}(\mbf{x}).
             \end{eqnarray*}
It is clear that $H_n(\mbf{x};0)=h_n$. The $H_n(\mbf{x};q)$ function encodes, as a Frobenius transform, the character of several interesting isomorphic graded $\S_n$-modules such as:  the coinvariant space of $\S_n$, the space of $\S_n$-harmonic polynomials, and the cohomology ring of the full-flag variety. More precisely, this makes explicit the graded decomposition into irreducibles of these spaces. Thus, the coefficient of $s_\mu(\mbf{x})$ in formula~\pref{defhnschur} corresponds to the Hilbert series\footnote{Graded dimension.} of the isotropic component of type $\mu$ of this space. Using \pref{cauchyp} to expand $H_n$, the global Hilbert series of these modules can be simply obtained by computing the scalar product
\begin{eqnarray}\label{hilbcoinv}
   \langle p_1^n,H_n\rangle&=&\sum_{\mu\vdash n} \langle p_1^n,s_\mu\rangle\, q^{n(\mu)}   [n]_q!\prod_{(i,j)\in \mu} \frac{1-q}{1-q^{h_{ij}}},\\
     &=&\sum_{\mu\vdash n}  \langle p_1^n,p_\mu\rangle \frac{p_\mu(1/(1-q))}{z_\mu}\prod_{k=1}^n (1-q^k)\\
     &=& \left(\frac{1}{1-q}\right)^n \prod_{k=1}^n (1-q^k)\\
     &=& [n]_q!.
      \end{eqnarray} 
To see this, recall that  $\langle p_\mu,p_\lambda\rangle$ is zero if $\mu\not=\lambda$, and $\langle p_\mu,p_\mu\rangle=z_\mu$. To complete the picture, let us also recall that $\langle p_\mu,s_\lambda\rangle$ is equal to the value, on the conjugacy class $\mu$, of the character of the irreducible representation associated to $\lambda$. In particular, it follows that
\begin{equation}
     H_n(\mbf{x};1)=p_1^n= \sum_{\mu\vdash n} f_\mu\, s_\mu(\mbf{x}).
  \end{equation}
This is the Frobenius characteristic of the regular representation of $\S_n$, for which the multiplicities $f_\mu$
are given by the number of standard Young tableaux of shape $\mu$.
The $H_n(\mbf{x};q)$ are special instances of the combinatorial Macdonald polynomials $H_\mu(\mbf{x};q,t)$ (not defined here, see~\cite{bergeron} for details). For instance, we have
\begin{eqnarray*}
   H_3(\mbf{x};q,t)&=&s_{{3}}(\mbf{x})+ \left( {q}^{2}+q \right) s_{{21}}(\mbf{x})+{q}^{3}s_{{111}}(\mbf{x}),\\
H_{21}(\mbf{x};q,t)&=&s_{{3}}(\mbf{x})+ \left( q+t \right) s_{{21}}(\mbf{x})+q\,t\,s_{{111}}(\mbf{x}),\\
H_{111}(\mbf{x};q,t)&=&s_{{3}}(\mbf{x})+ \left( {t}^{2}+t \right) s_{{21}}(\mbf{x})+{t}^{3}s_{{111}}(\mbf{x}).
\end{eqnarray*}
Beside this notion of Frobenius transform that ``formally'' encodes $\S_n$-irreducibles as Schur function, another more direct interpretation of the above formulas is in terms of characters of polynomial representations of $\GL(V)$, with $V$ an $N$-dimensional space over $\mathbb{C}$. Recall that the character, of a representation $\rho:\GL(V)\rightarrow \GL(W)$, is a symmetric function of $\chi_\rho(x_1,x_2,\ldots,x_N)$ of the eigenvalues of operators in $\GL(V)$.  Through Schur-Weyl duality, out of any $\S_n$-module $R$ and any $\GL(V)$-module $U$, one may construct a representation of $\GL(V)$:
          $$R(U):=R\otimes_{\C\S_n} U^{\otimes n},$$
 where $\S_n$ acts on $U^{\otimes n}$ by permutation of components. This construction is functorial:
      $$R:\hbox{$GL(V)$-Mod}\longrightarrow  \hbox{$GL(V)$-Mod},$$
and the character of $R(U)$ is the plethysm $f[g(x_1x_2,\ldots,x_N)]$, whenever $f$ is the Frobenius characteristic of $R$ and $g$ the character of $U$.  Furthermore, under this construction, irreducible polynomial representations of $\GL(V)$ correspond to irreducible $\S_n$-modules $R$. If such is the case, one writes $S^\lambda(V)$ when $R$ is irreducible of type $\lambda$. The corresponding character is the Schur function $s_\lambda(x_1,x_2,\ldots,x_N)$. For the special case $\lambda=(n)$, we get the symmetric power $S^a(V)$ whose character is $h_a(x_1,x_2,\ldots,x_N)$, hence the character of $S^a(S^b(V))$ is the plethysm $h_a[h_b]$.

\section*{Appendix B: Schur positivity}\label{AppendixB}
The proof of $\N$-positivity of the solution of the recurrence occurring in Proposition-\ref{proprec} may be directly translated in terms $\N$-positivity of the following, as a polynomial in $z$. Let us set
 \begin{equation}
     \rho(z;a):=\sum_{k=1}^\infty k\,a\binom{a+1}{2k+1}\, z^{2\,k+1},
 \end{equation}
and consider the following recurrence for $\theta_n(z)=\theta_n(z;a)$
 $$\theta_n(z)=\left( 3+z \right) \theta_{n-1}(z)+ \left( 1+z \right)  \left( z-3
 \right) \theta_{n-2}(z)+ \left( 1+z \right) ^{2} \left( 1-z
 \right) \theta_{n-3}(z),$$
 with initial conditions $\theta_0(z):=\rho(z;a-1)$, $\theta_1(z):=\rho(z;a)$, and
  $$\theta_2(z) = \sum_{k=1}^\infty\left( k(a-1)\binom{a+2}{2k+1}\, z^{2k+1} +
2k \binom{a+1}{2k+1}\, (1+z)\,z^{2\,k+1}\right).$$
For any $a>2$ (wth in $\N$), $\theta_{n}(z;a)$ is clearly a degree $a+n$ polynomials in the variable $z$, with positive integer coefficients,
and the link with our previous setup is simply that 
     $$\Theta_a(b) =h_2^b\, \theta_{b-a}(e_2/h_2;a).$$

\subsection*{General properties of Schur positivity}
Consider the ring $\Lambda_{\mbf{q}}$ of symmetric function with coefficients in the field of fractions $\Q(\mbf{q})$, with $\mbf{q}=q_1,q_2,\ldots$, in which we are interested in expansions in the Schur function (linear) basis.
We say that $f=f(\mbf{x};\mbf{q})$ in $\Lambda_{q,t}$ is \define{Schur-positive} if we have
    $$f(\mbf{x};\mbf{q})=\sum_{\lambda} a_\lambda(\mbf{q})\,s_\lambda(\mbf{x}),\qquad{\rm with}\qquad a_\lambda(\mbf{q})\in\N[\mbf{q}].$$
In other terms, for all partition $\lambda$, the coefficient $a_\lambda(\mbf{q})$ is a positive integer polynomial in the parameters $\mbf{q}$. If the difference $f-g$ is Schur positive, we  write $g\preceq_s f$. We have the following properties for this partial order on symmetric functions, whenever $f_1\preceq_s f_2$ and $g_1\preceq_s g_2$:
\begin{enumerate}\itemsep=4pt
\item[(1)] $f_1+g_1\preceq_s f_2+g_2$, evident from definition;
\item[(2)] $f_1\cdot g_1\preceq_s f_2\cdot g_2$, since products of Schur function expand positively in the Schur basis, with structure coefficients given by the Littlewood-Richardson rule;
\item[(3)] $f_2^\perp g_1\preceq_s f_1^\perp g_2$, where $f^\perp$ is the dual operator of multiplication by $f$ for the usual scalar product on symmetric functions (for which the Schur functions form an orthonormal basis);
\item[(4)] $f_1\circ g_1\preceq_s f_2\circ g_2$, since  plethysms of Schur functions expand positively in the Schur basis, as shown in Macdonald (see p.136, (8.10));
\item[(5)] $\overline{f_1}\cdot\overline{g_1}\preceq_s\overline{f_1\cdot g_1}$.

\end{enumerate}
  


\begin{thebibliography}{10}  

\bibitem{Abdesselam}
\auteur{A.~Abdesselam and J.~Chipalkatti},
\titreref{Brill-Gordan loci, transvectants and an analogue of the Foulkes conjecture}, Advances in Mathematics, 
Volume 208, Issue 2 (2007), 491--520.
See \href{http://arxiv.org/abs/math/0411110}{arXiv:math/0411110}.

\bibitem{bergeron}
\auteur{F.~Bergeron},
\titreref{Algebraic Combinatorics and Coinvariant Spaces}, CMS Treatise in Mathematics, CMS and A.K.Peters,  2009.


\bibitem{species}
\auteur{F.~Bergeron, P.~Leroux, and G.~Labelle},
\titreref{Combinatorial Species and Tree-Like structures},
Encyclopedia of Mathematics and its Applications \vol{67},
Cambridge University Press, 1998.


\bibitem{briand}
\auteur{E.~Briand},
\titreref{Polyn\^omes multisym\'etriques}, PhD dissertation, University Rennes I, Rennes, France, October 2002.

\bibitem{brion}
\auteur{M.~Brion},
\titreref{Stable properties of plethysm: on two conjectures of Foulkes},
Manuscripta Mathematica \vol{80} (1993), 347--371.

\bibitem{deBoeck}
\auteur{M.~de Boeck},
\titreref{A study of Foulkes modules using semistandard homomorphisms}, 	
see \href{http://arxiv.org/abs/1409.0734}{arXiv:1409.0734}.


\bibitem{doran}
\auteur{W.~F.~Doran IV},
\titreref{On Foulke's conjecture},
Journal of Pure and Applied Algbra \vol{130} (1998), 85--98.

\bibitem{cheung}
\auteur{M.-W.~Cheung, C.~Ikenmeyer, and S.~Mkrtchyan},
\titreref{Symmetrizing Tableaux and the 5th case of the Foulkes Conjecture}. See \href{http://arxiv.org/abs/1509.03944}{arXiv:1509.03944}. 

\bibitem{colmenarejo}
\auteur{L.~Colmenarejo},
\titreref{Stability Properties of the Plethysm: a Combinatorial Approach}, DMTCS Proceedings FPSAC'15 (2015) 877--888.
See \href{http://arxiv.org/abs/1505.03842}{arXiv:1505.03842}. 


\bibitem{foulkes}
\auteur{H.~O.~Foulkes},
\titreref{Concomitants of the quintic and sextic up to degree four in the coefficients of the ground form},
Journal of the London Mathematical Society \vol{25} (1950) 205--209.

\bibitem{hermite}
\auteur{C.~Hermite},
\titreref{Sur la th\'eorie des fonctions homog\`enes \`a deux ind\'etermin\'ees},
Cambridge and Dublin Mathematical Journal \vol{9} (1854) 172Ð 217.

\bibitem{littlewood}
\auteur{D.~E.~Littlewood},
\titreref{Invariant theory, tensors and group characters},
Philosophical Transactions of the Royal Society of London (A) (1944), 305--365.

\bibitem{macdonald} 
\auteur{I.~G. Macdonald},
\titreref{Symmetric functions and Hall polynomials}, second ed., Oxford
  Mathematical Monographs, The Clarendon Press Oxford University Press, New York, 1995.
  
\bibitem{manivelsylvester}
\auteur{L.~Manivel},
\titreref{An extension of the Cayley-Sylvester formula},
European Journal of Combinatorics \vol{28} (2007) 1839--1842.

\bibitem{manivelgaussian}
\auteur{L.~Manivel},
\titreref{Gaussian maps and plethysm},
in Algebraic geometry (Catania, 1993/Barcelona, 1994), volume 200 of Lecture Notes in Pure and Applied Mathematics, pages 91--117. Dekker, New York, 1998.

\bibitem{mckay}
\auteur{T.~McKay},
\titreref{On plethysm conjectures of Stanley and Foulkes},
 Journal of Algebra \vol{319(5)} (2008) 2050--2071.
 
 \bibitem{paget}
 \auteur{R.~Paget and M.~Wildon},
 \titreref{Set Families and Foulkes Modules},
Journal of Algebraic Combinatorics \vol{34} Issue 3 (2011) 525--544.
See \href{http://arxiv.org/abs/1007.2946v2}{arXiv:1007.2946v2}.
 
 \bibitem{pylyavskyy}
 \auteur{P.~Pylyavskyy},
 \titreref{On Plethysm Conjectures of Stanley and Foulkes: the $2\times n$ Case},
 The Electronic Journal of combinatorics \vol{11}, Issue 2 (2004-6) (The Stanley Festschrift volume), \href{http://www.combinatorics.org/ojs/index.php/eljc/article/view/v11i2r8}{\#R8}.
 
 
 \bibitem{sivek}
 \auteur{S.~Sivek},
 \titreref{Some plethysm results related to Foulkes' conjecture},
 The Electronic Journal of combinatorics \vol{13} (2006), \href{http://www.combinatorics.org/ojs/index.php/eljc/article/view/v13i1r24}{\#R24}.
 
 \bibitem{vessenes}
 \auteur{R.~Vessenes},
 \titreref{Generalized Foulkes' Conjecture and tableaux},
 Journal of Algebra \vol{277} (2004) 579--614.
 
 \bibitem{weintraub}
 \auteur{S.~H.~Weintraub}, 
 \titreref{Some observations on plethysms},
 Journal of Algebra \vol{129} (1990), 103--114.

\end{thebibliography}
\end{document}